\documentclass[12pt,twoside]{amsart}
\usepackage[a4paper,margin=1in]{geometry}
\usepackage[T1]{fontenc}
\usepackage{lmodern}
\usepackage{microtype}
\usepackage{amsmath,amssymb,amsthm,mathtools,amscd}
\usepackage{booktabs,tabularx,array}
\usepackage{enumitem}
\usepackage[hidelinks]{hyperref}
\usepackage{cleveref}
\usepackage{mathrsfs}

\date{}
\allowdisplaybreaks[4] \footskip=15pt
\renewcommand{\uppercasenonmath}[1]{}

\numberwithin{equation}{section} \theoremstyle{plain}
\newtheorem{theorem}{Theorem}[section]
\newtheorem{lemma}[theorem]{Lemma}
\newtheorem{proposition}[theorem]{Proposition}
\newtheorem{corollary}[theorem]{Corollary}
\theoremstyle{definition}
\newtheorem{definition}[theorem]{Definition}
\newtheorem{example}[theorem]{Example}
\theoremstyle{remark}
\newtheorem{remark}[theorem]{Remark}
\newtheorem*{ack*}{ACKNOWLEDGEMENTS}

\newcommand{\Mod}{\operatorname{Mod}}
\newcommand{\Spec}{\operatorname{Spec}}
\newcommand{\Max}{\operatorname{Max}}
\newcommand{\rad}{\operatorname{rad}}
\newcommand{\Soc}{\operatorname{Soc}}
\newcommand{\Top}{\operatorname{Top}}
\newcommand{\End}{\operatorname{End}}
\newcommand{\Coker}{\operatorname{Coker}}
\newcommand{\Ker}{\operatorname{Ker}}
\newcommand{\im}{\operatorname{im}}

\newcommand{\Ext}{\operatorname{Ext}}
\newcommand{\Hom}{\operatorname{Hom}}

\newcommand{\op}{\mathrm{op}}

\newcommand{\smallin}{\ll}

\begin{document}
\begin{center}
{\large  \bf Duality Between Injective Envelopes and Flat Covers over Noether Algebras}
	
	\vspace{0.5cm}  Xiaolei Zhang\\

	{\footnotesize  School of Mathematics and Statistics,  Tianshui Normal University,  Tianshui 741001, China

	E-mail addresses: zxlrghj@163.com\\}
\end{center}

\bigskip
\centerline { \bf  Abstract}
\bigskip
\leftskip10truemm \rightskip10truemm \noindent

It follows by  Puuska  that, over a commutative Noetherian ring $R$, a morphism $i : M \rightarrow I$ of $R$-modules is an
injective envelope if and only if its Matlis dual $\Hom_R(i, E)$ is a flat cover for some injective
cogenerator $E$, and equivalently for every injective $R$-module $E$. In this paper, we will generalize this result to non-commutative Noetherian algebras. \\
\vbox to 0.3cm{}

\leftskip0truemm \rightskip0truemm
\bigskip

\section{Introduction}
Let $R$ be a commutative noetherian ring. Matlis's decomposition theory for injective modules \cite{Matlis1958} and the subsequent theory of flat cotorsion modules initiated by Enochs \cite{Enochs1984} form the classical background for this paper. Puuska proved that a morphism $i:M\to I$ of $R$-modules is an injective envelope if and only if its Matlis dual $\Hom_R(i,E)$ is a flat cover for some injective cogenerator $E$, and equivalently for every injective $R$-module $E$; see \cite[Theorem~1.1]{Puuska2024}. The essential mechanism in that proof is local. Injective envelopes are detected by residue-field socles, flat covers of cotorsion modules are detected by the corresponding local top functors, and a tensor--Hom duality identifies the two tests. 

At the preenvelope/precover level there is a more general character-dual phenomenon over one-sided noetherian rings; see Enochs--Huang \cite{EnochsHuang2012}. This belongs to the broader character-module tradition originating with Lambek's flatness criterion \cite{Lambek1964}. Passing from precovers to covers, however, requires a minimality argument. The purpose here is therefore not merely to dualize exactness, but to prove that minimality is preserved in the required direction. In fact, we want to establish the same envelope--cover duality for a possibly noncommutative Noether algebra.

 Let $A$ be a Noether $R$-algebra. The appropriate duality is not an $A$-linear Hom functor, but the central dual
\[
D_E(-)=\Hom_R(-,E),
\]
which changes right $A$-modules into left $A$-modules. This is precisely the pointwise Matlis-duality framework used by Kanda and Nakamura in their structure theory of flat cotorsion modules over Noether algebras \cite{KandaNakamura2022}.

The noncommutative argument requires two replacements for the residue field $\kappa(\mathfrak p)$ occurring in the commutative proof. For $\mathfrak p\in\Spec R$, write
\[
A_{\mathfrak p}=A\otimes_R R_{\mathfrak p},\qquad
J_{\mathfrak p}=\rad A_{\mathfrak p},\qquad
\overline A_{\mathfrak p}=A_{\mathfrak p}/J_{\mathfrak p}.
\]
Kanda--Nakamura show that $\overline A_{\mathfrak p}$ is a finite-dimensional semisimple $\kappa(\mathfrak p)$-algebra. Thus the local socle of a right $A$-module $X$ is naturally measured by
\[
S_{\mathfrak p}(X)=\Hom_{A_{\mathfrak p}}(\overline A_{\mathfrak p},X_{\mathfrak p}),
\]
while the dual local top of a left $A$-module $Y$ is measured by
\[
T_{\mathfrak p}(Y)=\overline A_{\mathfrak p}\otimes_{A_{\mathfrak p}}\Hom_R(R_{\mathfrak p},Y).
\]
The first main technical point is the canonical natural isomorphism
\[
T_{\mathfrak p}(D_E X)\cong D_E(S_{\mathfrak p}(X)).
\]
The second is that the family $\{T_{\mathfrak p}\}_{\mathfrak p\in\Spec R}$ detects redundant direct summands in a flat cotorsion source. This is where the Kanda--Nakamura decomposition of flat cotorsion modules into local complete pieces enters.

The following main result is to generalize the Puuska's theorem to non-commutative Noether algebras utilizing Kanda--Nakamura decomposition of flat cotorsion modules: \begin{theorem}$($=Theorem 6.1$)$
	Let $R$ be a commutative noetherian ring, let $A$ be a Noether $R$-algebra, and let $i:M\to I$ be a morphism of right $A$-modules. The following conditions are equivalent:
	\begin{enumerate}[label=(\arabic*)]
		\item $i$ is an injective envelope in $\Mod A$;
		\item for some injective cogenerator $E$ of $\Mod R$, the morphism
		\[
		D_E(i)=\Hom_R(i,E):\Hom_R(I,E)\longrightarrow\Hom_R(M,E)
		\]
		is a flat cover in $\Mod A^{\op}$;
		\item for every injective $R$-module $E$, the morphism $D_E(i)$ is a flat cover in $\Mod A^{\op}$.
	\end{enumerate}
\end{theorem}
Standard background on covers, cotorsion pairs, and relative homological algebra may be found in \cite{EnochsJenda2000,Xu1996}.

Throughout the paper all rings are associative with identity and all modules are unital. The symbols used throughout are collected here.

\begin{center}
	\small
	\begin{tabularx}{\textwidth}{@{}>{\raggedright\arraybackslash}p{0.24\textwidth}X@{}}
		\toprule
		\textbf{Symbol} & \textbf{Meaning} \\
		\midrule
		$R,A,Z(A)$ & the commutative noetherian base ring, a Noether $R$-algebra, and the center of $A$ \\
		$\Mod A,\Mod A^{\op}$ & right $A$-modules and, respectively, left $A$-modules \\
		$\Spec R,\Spec A,\Max A$ & prime ideals of $R$, two-sided prime ideals of $A$, and maximal two-sided ideals of $A$ \\
		$\mathfrak p,P$ & a prime of $R$ and a prime of $A$; when both occur, $\mathfrak p=P\cap R$ \\
		$R_{\mathfrak p},A_{\mathfrak p},X_{\mathfrak p}$ & central localizations of $R$, $A$, and an $R$-module or $A$-module $X$ at $R\setminus\mathfrak p$ \\
		$\kappa(\mathfrak p)$ & the residue field $R_{\mathfrak p}/\mathfrak pR_{\mathfrak p}$ \\
		$J_{\mathfrak p},\overline A_{\mathfrak p}$ & $\rad A_{\mathfrak p}$ and the semisimple quotient $A_{\mathfrak p}/J_{\mathfrak p}$ \\
		$E_A(X)$ & the injective envelope of the right $A$-module $X$ \\
		$S_A(P),I_A(P)$ & the simple module attached to $P$ and its injective envelope $E_A(S_A(P))$ \\
		$D_E(-),D(-)$ & $\Hom_R(-,E)$ and, in Appendix~A, the standard Artin dual $\Hom_R(-,E_R(R/J(R)))$ \\
		$S_{\mathfrak p},T_{\mathfrak p}$ & the local socle and local top functors defined in \eqref{eq:Sp} and \eqref{eq:Tp} \\
		$\Lambda^{\mathfrak p}X$ & the $\mathfrak p$-adic completion $\varprojlim_{n\ge1}X/\mathfrak p^nX$ \\
		$\Soc,\Top,\rad$ & socle, top, and Jacobson radical \\
		$\Hom,\Ext,\End$ & homomorphism, extension, and endomorphism groups over the displayed subscript ring \\
		$\Ker,\Coker,\im$ & kernel, cokernel, and image \\
		$K\smallin P$ & $K$ is superfluous (small) in $P$: $K+L=P$ implies $L=P$ \\
		\bottomrule
	\end{tabularx}
\end{center}

\section{Preliminaries on Noether algebras and central duality}
\begin{definition}
	Let $R$ be a commutative noetherian ring. A \emph{Noether $R$-algebra} is a ring $A$ equipped with a ring homomorphism $R\to A$ whose image is contained in $Z(A)$ and such that $A$ is finitely generated as an $R$-module.
\end{definition}

A Noether $R$-algebra is left and right noetherian. We write $\Mod A$ for the category of right $A$-modules and identify $\Mod A^{\op}$ with the category of left $A$-modules. For $\mathfrak p\in\Spec R$ set
\[
A_{\mathfrak p}=A\otimes_R R_{\mathfrak p},\qquad
J_{\mathfrak p}=\rad A_{\mathfrak p},\qquad
\overline A_{\mathfrak p}=A_{\mathfrak p}/J_{\mathfrak p}.
\]
By \cite[Proposition~2.19 and Remark~2.20]{KandaNakamura2022}, $\overline A_{\mathfrak p}$ is a finite-dimensional semisimple $\kappa(\mathfrak p)$-algebra and
\begin{equation}\label{eq:pJ}
	\mathfrak p A_{\mathfrak p}\subseteq J_{\mathfrak p}.
\end{equation}
The $\mathfrak p$-adic completion notation is always taken with respect to the central ideal $\mathfrak pA$, and $\Hom_R(R_{\mathfrak p},-)$ denotes the corresponding colocalization functor.

We recall the approximation terminology, following \cite{Enochs1984,Xu1996}. A morphism $f:F\to M$ with $F$ flat is a flat precover if every morphism from a flat module to $M$ factors through $f$. A flat precover is a flat cover if it is right minimal, that is, if every $u\in\End_A(F)$ satisfying $fu=f$ is an automorphism. Flat covers exist over arbitrary rings \cite{BicanElBashirEnochs2001}; moreover a flat precover is necessarily an epimorphism, since every morphism $A\to M$ factors through it. A right $A$-module $C$ is cotorsion if
\[
\Ext_A^1(F,C)=0
\]
for every flat right $A$-module $F$.

Fix an injective $R$-module $E$. For a right $A$-module $X$ define
\[
D_E(X)=\Hom_R(X,E)
\]
with left $A$-action
\[
(a\varphi)(x)=\varphi(xa).
\]
Thus $D_E:\Mod A\to\Mod A^{\op}$ is exact and contravariant. We shall use repeatedly the following consequences of \cite[Proposition~2.2]{KandaNakamura2022}:
\begin{enumerate}[label=(D\arabic*)]
	\item $D_E(X)$ is pure-injective, hence cotorsion, for every $X$;
	\item if $X$ is injective, then $D_E(X)$ is flat and cotorsion;
	\item if $E$ is an injective cogenerator and $D_E(X)$ is flat, then $X$ is injective.
\end{enumerate}
The same assertions hold after replacing $A$ by $A^{\op}$.

We also recall three structural facts from Kanda--Nakamura \cite{KandaNakamura2022}. First, every flat cotorsion right $A$-module $H$ admits a product decomposition
\begin{equation}\label{eq:prod}
	H\cong\prod_{\mathfrak q\in\Spec R}H(\mathfrak q),
\end{equation}
where each $H(\mathfrak q)$ is $\mathfrak q$-local, $\mathfrak q$-adically complete, flat, and cotorsion; the isomorphism class of $H(\mathfrak q)$ is uniquely determined by $H$ \cite[Proposition~3.7]{KandaNakamura2022}. Moreover,
\begin{equation}\label{eq:completionpiece}
	\Lambda^{\mathfrak p}\Hom_R(R_{\mathfrak p},H)\cong H(\mathfrak p)
\end{equation}
for every $\mathfrak p\in\Spec R$ \cite[Lemma~3.3(3)]{KandaNakamura2022}. Second, if $F$ is a $\mathfrak p$-local, $\mathfrak p$-complete flat right $A$-module, then the canonical map
\begin{equation}\label{eq:localcover}
	F\longrightarrow F\otimes_{A_{\mathfrak p}}\overline A_{\mathfrak p}
\end{equation}
is a flat cover \cite[Proposition~4.5(1)]{KandaNakamura2022}. Third, the indecomposable injective right $A$-modules are indexed by $\Spec A$: for $P\in\Spec A$, with $\mathfrak p=P\cap R$, the module
\[
I_A(P)=E_A(S_A(P))
\]
is $\mathfrak p$-local and every indecomposable injective right $A$-module is of this form; see \cite[(2.5)--(2.6) and Theorem~2.22]{KandaNakamura2022}.

\section{Local socles and injective envelopes}
For $X\in\Mod A$ and $\mathfrak p\in\Spec R$ define
\begin{equation}\label{eq:Sp}
	S_{\mathfrak p}(X)=\Hom_{A_{\mathfrak p}}(\overline A_{\mathfrak p},X_{\mathfrak p}).
\end{equation}
Since $\overline A_{\mathfrak p}$ is semisimple, this functor measures the semisimple $\overline A_{\mathfrak p}$-socle of $X_{\mathfrak p}$ relevant to the prime $\mathfrak p$.

We first prove that central localization preserves essentiality. The noetherian hypothesis is needed because the ambient injective module need not be finitely generated.

\begin{lemma}\label{lem:local-essential}
	Let $A$ be right noetherian, let $S\subseteq Z(A)$ be a multiplicatively closed subset, and let $N\subseteq M$ be an essential right $A$-submodule. Then $S^{-1}N\subseteq S^{-1}M$ is essential as an $S^{-1}A$-submodule.
\end{lemma}
\begin{proof}
	Let $0\ne L\subseteq S^{-1}M$ be an $S^{-1}A$-submodule. Choose $0\ne x/s\in L$. Since $s/1$ is a unit of $S^{-1}A$,
	\[
	(x/s)(s/1)=x/1\ne0,
	\]
	so $x/1\in L$. Put $U=xA$. The cyclic module $U$ is noetherian because $A$ is right noetherian, and $S^{-1}U$ is a nonzero submodule of $L$. Let
	\[
	V=N\cap U.
 	\]
	Since $N$ is essential in $M$, the submodule $V$ is essential in $U$.
	
	We claim that $S^{-1}V\ne0$. Suppose to the contrary that $S^{-1}V=0$. Because $V$ is finitely generated, choose generators $v_1,\dots,v_r$. For each $i$ there exists $s_i\in S$ with $v_is_i=0$. Since $S\subseteq Z(A)$, the product
	\[
	t=s_1\cdots s_r\in S
	\]
	annihilates all of $V$. Let $\mu_t:U\to U$ denote right multiplication by $t$. Then $V\subseteq\Ker\mu_t$. Since $V$ is essential in $U$, the submodule $\Ker\mu_t$ is essential in $U$.
	
	Because $U$ is noetherian, the ascending chain
	\[
	\Ker\mu_t\subseteq\Ker\mu_{t^2}\subseteq\Ker\mu_{t^3}\subseteq\cdots
	\]
	stabilizes, say $\Ker\mu_{t^n}=\Ker\mu_{t^{n+1}}$ for some $n\ge1$. If $Ut^n\ne0$, then the nonzero submodule $Ut^n$ meets the essential submodule $\Ker\mu_t$ nontrivially. Hence there is $u\in U$ such that
	\[
	0\ne ut^n\in\Ker\mu_t.
	\]
	Thus $ut^{n+1}=0$, so $u\in\Ker\mu_{t^{n+1}}=\Ker\mu_{t^n}$, contradicting $ut^n\ne0$. Therefore $Ut^n=0$. But $t$ is invertible after localizing at $S$, so $S^{-1}U=0$, contradicting the construction of $U$. Hence $S^{-1}V\ne0$.
	
	Finally,
	\[
	0\ne S^{-1}V\subseteq (S^{-1}N)\cap L.
	\]
	Thus every nonzero submodule of $S^{-1}M$ meets $S^{-1}N$ nontrivially, proving essentiality.
\end{proof}

\begin{proposition}[Local criterion for injective envelopes]\label{prop:local-inj}
	Let $i:M\to I$ be a morphism in $\Mod A$. The following are equivalent:
	\begin{enumerate}[label=(\roman*)]
		\item $i$ is an injective envelope;
		\item $I$ is injective, $i$ is a monomorphism, and
		\[
		S_{\mathfrak p}(i):\Hom_{A_{\mathfrak p}}(\overline A_{\mathfrak p},M_{\mathfrak p})
		\longrightarrow
		\Hom_{A_{\mathfrak p}}(\overline A_{\mathfrak p},I_{\mathfrak p})
		\]
		is an isomorphism for every $\mathfrak p\in\Spec R$.
	\end{enumerate}
\end{proposition}
\begin{proof}
	Assume first that $i$ is an injective envelope. Then $I$ is injective and $M\subseteq I$ is essential. By Lemma~\ref{lem:local-essential}, $M_{\mathfrak p}\subseteq I_{\mathfrak p}$ is essential for every $\mathfrak p\in\Spec R$. Since $i_{\mathfrak p}$ is a monomorphism and $\Hom_{A_{\mathfrak p}}(\overline A_{\mathfrak p},-)$ is left exact, $S_{\mathfrak p}(i)$ is injective.
	
	To prove surjectivity, let
	\[
	\alpha:\overline A_{\mathfrak p}\longrightarrow I_{\mathfrak p}
	\]
	be an $A_{\mathfrak p}$-homomorphism and put $L=\im\alpha$. The module $L$ is semisimple because it is a quotient of the semisimple right $A_{\mathfrak p}$-module $\overline A_{\mathfrak p}$. Essentiality of $M_{\mathfrak p}$ in $I_{\mathfrak p}$ implies that $L\cap M_{\mathfrak p}$ is essential in $L$. A semisimple module has no proper essential submodule: if $U\subsetneq L$ is a submodule, semisimplicity gives $L=U\oplus U'$ with $U'\ne0$, and then $U\cap U'=0$. Therefore $L\cap M_{\mathfrak p}=L$, so $L\subseteq M_{\mathfrak p}$. Hence $\alpha$ factors through $M_{\mathfrak p}$, proving that $S_{\mathfrak p}(i)$ is surjective.
	
	Conversely, assume (ii). Let
	\[
	e:M\longrightarrow E_A(M)
	\]
	be the injective envelope of $M$. Since $I$ is injective, $i$ extends across $e$: there exists an $A$-homomorphism
	\[
	u:E_A(M)\longrightarrow I
	\]
	with $ue=i$. We first show that $u$ is a monomorphism. If $\Ker u\ne0$, essentiality of $e(M)$ in $E_A(M)$ gives
	\[
	0\ne\Ker u\cap e(M).
	\]
	But $u$ is injective on $e(M)$ because $ue=i$ and $i$ is monic, a contradiction. Thus $u$ is monic. Since $E_A(M)$ is injective, $u$ splits, so
	\begin{equation}\label{eq:Isplit}
		I=u(E_A(M))\oplus I'
	\end{equation}
	for some injective right $A$-module $I'$.
	
	Suppose $I'\ne0$. Since $A$ is right noetherian, every injective right $A$-module decomposes as a direct sum of indecomposable injectives (see \cite[Theorem~2.5]{Matlis1958}). Hence $I'$ has a direct summand isomorphic to $I_A(P)$ for some $P\in\Spec A$. Put $\mathfrak p=P\cap R$. The module $I_A(P)$ is $\mathfrak p$-local by \cite[(2.6)]{KandaNakamura2022}. Moreover, by \cite[Proposition~2.19]{KandaNakamura2022}, the semisimple right $A_{\mathfrak p}$-module $\overline A_{\mathfrak p}$ contains a nonzero direct summand isomorphic to a finite direct sum of copies of $S_A(P)$. Therefore there is a nonzero composite
	\[
	\overline A_{\mathfrak p}\longrightarrow S_A(P)\longrightarrow I_A(P),
	\]
	and consequently
	\[
	\Hom_{A_{\mathfrak p}}(\overline A_{\mathfrak p},I'_{\mathfrak p})\ne0.
	\]
	After localizing \eqref{eq:Isplit}, we have
	\[
	I_{\mathfrak p}=u(E_A(M))_{\mathfrak p}\oplus I'_{\mathfrak p}.
	\]
	Since $i_{\mathfrak p}(M_{\mathfrak p})\subseteq u(E_A(M))_{\mathfrak p}$, the image of $S_{\mathfrak p}(i)$ is contained in the first summand of
	\[
	\Hom_{A_{\mathfrak p}}(\overline A_{\mathfrak p},I_{\mathfrak p})
	\cong
	\Hom_{A_{\mathfrak p}}(\overline A_{\mathfrak p},u(E_A(M))_{\mathfrak p})
	\oplus
	\Hom_{A_{\mathfrak p}}(\overline A_{\mathfrak p},I'_{\mathfrak p}).
	\]
	The second summand is nonzero, contradicting the assumed surjectivity of $S_{\mathfrak p}(i)$. Thus $I'=0$, so $u$ is an isomorphism and $i$ is an injective envelope.
\end{proof}

\section{The local socle--top duality}
For a left $A$-module $Y$ and $\mathfrak p\in\Spec R$, define
\begin{equation}\label{eq:Tp}
	T_{\mathfrak p}(Y)=\overline A_{\mathfrak p}\otimes_{A_{\mathfrak p}}\Hom_R(R_{\mathfrak p},Y),
\end{equation}
where $\Hom_R(R_{\mathfrak p},Y)$ carries its canonical left $A_{\mathfrak p}$-module structure. The next lemma is the exact noncommutative analogue of the tensor--Hom identification used by Puuska.

\begin{lemma}\label{lem:socle-top}
	Let $E$ be an injective $R$-module, let $X\in\Mod A$, and let $\mathfrak p\in\Spec R$. There is a natural isomorphism of $R$-modules
	\begin{equation}\label{eq:local-dual}
		T_{\mathfrak p}(D_E X)\xrightarrow{\ \cong\ }D_E(S_{\mathfrak p}(X)).
	\end{equation}
	It is natural in $X$ and in $E$.
\end{lemma}
\begin{proof}
	Tensor--Hom adjunction over the commutative ring $R$ gives a natural isomorphism
	\begin{equation}\label{eq:coloc-dual}
		\Hom_R(R_{\mathfrak p},\Hom_R(X,E))\xrightarrow{\ \cong\ }\Hom_R(X\otimes_RR_{\mathfrak p},E)=\Hom_R(X_{\mathfrak p},E).
	\end{equation}
	Explicitly, if $\eta:R_{\mathfrak p}\to\Hom_R(X,E)$, then the corresponding map $X_{\mathfrak p}=X\otimes_RR_{\mathfrak p}\to E$ is
	\[
	x\otimes q\longmapsto\eta(q)(x).
	\]
	With the canonical left $A_{\mathfrak p}$-module structures, \eqref{eq:coloc-dual} is $A_{\mathfrak p}$-linear. Thus it remains to prove
	\begin{equation}\label{eq:eval-general}
		\overline A_{\mathfrak p}\otimes_{A_{\mathfrak p}}\Hom_R(X_{\mathfrak p},E)
		\xrightarrow{\ \cong\ }
		\Hom_R\!\left(\Hom_{A_{\mathfrak p}}(\overline A_{\mathfrak p},X_{\mathfrak p}),E\right).
	\end{equation}
	
	We prove a slightly more general statement. Let $B$ be a finitely presented right $A_{\mathfrak p}$-module. Define
	\[
	\Theta_{B,X}:B\otimes_{A_{\mathfrak p}}\Hom_R(X_{\mathfrak p},E)
	\longrightarrow
	\Hom_R\!\left(\Hom_{A_{\mathfrak p}}(B,X_{\mathfrak p}),E\right)
	\]
	by
	\begin{equation}\label{eq:theta}
		\Theta_{B,X}(b\otimes\lambda)(g)=\lambda(g(b)).
	\end{equation}
	We verify first that this is well defined. If $a\in A_{\mathfrak p}$, then
	\[
	\Theta_{B,X}(ba\otimes\lambda)(g)=\lambda(g(ba))=\lambda(g(b)a)=(a\lambda)(g(b))
	=\Theta_{B,X}(b\otimes a\lambda)(g),
	\]
	where the left $A_{\mathfrak p}$-action on $\Hom_R(X_{\mathfrak p},E)$ is $(a\lambda)(x)=\lambda(xa)$. Hence \eqref{eq:theta} is $A_{\mathfrak p}$-balanced. Its naturality in $B,X$, and $E$ is immediate from the evaluation formula.
	
	If $B=A_{\mathfrak p}$, then $\Theta_{B,X}$ is the standard identification
	\[
	A_{\mathfrak p}\otimes_{A_{\mathfrak p}}\Hom_R(X_{\mathfrak p},E)
	\cong\Hom_R(X_{\mathfrak p},E)
	\cong\Hom_R(\Hom_{A_{\mathfrak p}}(A_{\mathfrak p},X_{\mathfrak p}),E).
	\]
	Therefore $\Theta_{B,X}$ is an isomorphism whenever $B$ is a finite free right $A_{\mathfrak p}$-module.
	
	Now choose a finite presentation
	\[
	A_{\mathfrak p}^{m}\xrightarrow{\delta}A_{\mathfrak p}^{n}\longrightarrow B\longrightarrow0.
	\]
	Applying $\Hom_{A_{\mathfrak p}}(-,X_{\mathfrak p})$ gives an exact sequence
	\[
	0\longrightarrow\Hom_{A_{\mathfrak p}}(B,X_{\mathfrak p})\longrightarrow X_{\mathfrak p}^{n}\longrightarrow X_{\mathfrak p}^{m}.
	\]
	Since $E$ is injective over $R$, the contravariant functor $\Hom_R(-,E)$ is exact, and hence we obtain a right-exact sequence
	\[
	\Hom_R(X_{\mathfrak p}^{m},E)\longrightarrow\Hom_R(X_{\mathfrak p}^{n},E)
	\longrightarrow\Hom_R(\Hom_{A_{\mathfrak p}}(B,X_{\mathfrak p}),E)\longrightarrow0.
	\]
	On the other hand, tensoring the presentation with the left $A_{\mathfrak p}$-module $\Hom_R(X_{\mathfrak p},E)$ yields
	\[
	\Hom_R(X_{\mathfrak p},E)^m\longrightarrow\Hom_R(X_{\mathfrak p},E)^n
	\longrightarrow B\otimes_{A_{\mathfrak p}}\Hom_R(X_{\mathfrak p},E)\longrightarrow0.
	\]
	Under the canonical identifications $\Hom_R(X_{\mathfrak p}^{r},E)\cong\Hom_R(X_{\mathfrak p},E)^r$, the first maps in the last two sequences coincide. This follows directly from the definition $(a\lambda)(x)=\lambda(xa)$ of the left $A_{\mathfrak p}$-action. Therefore the induced map on cokernels is precisely $\Theta_{B,X}$ and is an isomorphism.
	
	Finally, $A_{\mathfrak p}$ is right noetherian, so the cyclic right $A_{\mathfrak p}$-module $\overline A_{\mathfrak p}=A_{\mathfrak p}/J_{\mathfrak p}$ is finitely presented. Taking $B=\overline A_{\mathfrak p}$ gives \eqref{eq:eval-general}; combining this with \eqref{eq:coloc-dual} proves \eqref{eq:local-dual}.
\end{proof}

\begin{remark}
	If $A=R$ is commutative, then $A_{\mathfrak p}=R_{\mathfrak p}$, $J_{\mathfrak p}=\mathfrak pR_{\mathfrak p}$, and $\overline A_{\mathfrak p}=\kappa(\mathfrak p)$. Thus Lemma~\ref{lem:socle-top} becomes Puuska's natural isomorphism
	\[
	\kappa(\mathfrak p)\otimes_{R_{\mathfrak p}}\Hom_R(R_{\mathfrak p},\Hom_R(X,E))
	\cong
	\Hom_R\!\left(\Hom_{R_{\mathfrak p}}(\kappa(\mathfrak p),X_{\mathfrak p}),E\right).
	\]
\end{remark}

\section{A local minimality detector for flat precovers}
The next lemma is the point at which the structure theorem for flat cotorsion modules over Noether algebras replaces the commutative flat-cover criterion used by Puuska.

\begin{lemma}\label{lem:nonvanishing-top}
	Let $0\ne H$ be a flat cotorsion left $A$-module. Then $T_{\mathfrak p}(H)\ne0$ for some $\mathfrak p\in\Spec R$.
\end{lemma}
\begin{proof}
	Regard $H$ as a right $A^{\op}$-module. Since $A^{\op}$ is again a Noether $R$-algebra, \cite[Proposition~3.7]{KandaNakamura2022} gives a product decomposition
	\[
	H\cong\prod_{\mathfrak q\in\Spec R}H(\mathfrak q),
	\]
	where every $H(\mathfrak q)$ is $\mathfrak q$-local, $\mathfrak q$-complete, flat, and cotorsion. Since $H\ne0$, choose $\mathfrak p$ with $H(\mathfrak p)\ne0$. Put
	\[
	C=\Hom_R(R_{\mathfrak p},H).
	\]
	By \cite[Lemma~3.3(3)]{KandaNakamura2022},
	\begin{equation}\label{eq:completion-Hp}
		\Lambda^{\mathfrak p}C\cong H(\mathfrak p).
	\end{equation}
	The completion map $C\to\Lambda^{\mathfrak p}C$ induces an isomorphism modulo $\mathfrak p$. Indeed, \cite[Lemma~A.4]{KandaNakamura2022}, applied with $\mathfrak a=\mathfrak b=\mathfrak p$, gives
	\begin{equation}\label{eq:modp-completion}
		C/\mathfrak pC\xrightarrow{\ \cong\ }\Lambda^{\mathfrak p}C/\mathfrak p\Lambda^{\mathfrak p}C.
	\end{equation}
	By \eqref{eq:pJ}, $\mathfrak p$ annihilates the right $A_{\mathfrak p}$-module $\overline A_{\mathfrak p}$. Hence tensoring with $\overline A_{\mathfrak p}$ over $A_{\mathfrak p}$ factors through reduction modulo $\mathfrak p$, and \eqref{eq:modp-completion} yields
	\begin{equation}\label{eq:Tidentify}
		\overline A_{\mathfrak p}\otimes_{A_{\mathfrak p}}C
		\xrightarrow{\ \cong\ }
		\overline A_{\mathfrak p}\otimes_{A_{\mathfrak p}}\Lambda^{\mathfrak p}C
		\xrightarrow{\ \cong\ }
		\overline A_{\mathfrak p}\otimes_{A_{\mathfrak p}}H(\mathfrak p),
	\end{equation}
	where the second isomorphism uses \eqref{eq:completion-Hp}. The left-hand side of \eqref{eq:Tidentify} is $T_{\mathfrak p}(H)$.
	
	It remains to prove that the right-hand side is nonzero. Apply \cite[Proposition~4.5(1)]{KandaNakamura2022} to the Noether $R$-algebra $A^{\op}$. Since $(A^{\op})_{\mathfrak p}\cong A_{\mathfrak p}^{\op}$ and $\rad(A_{\mathfrak p}^{\op})=J_{\mathfrak p}^{\op}$, the canonical semisimple quotient is $\overline A_{\mathfrak p}^{\op}$. In right $A^{\op}$-module notation the resulting flat cover is
	\[
	H(\mathfrak p)\longrightarrow H(\mathfrak p)\otimes_{A_{\mathfrak p}^{\op}}\overline A_{\mathfrak p}^{\op}.
	\]
	Using the standard identification
	\[
	H(\mathfrak p)\otimes_{A_{\mathfrak p}^{\op}}\overline A_{\mathfrak p}^{\op}
	\xrightarrow{\ \cong\ }
	\overline A_{\mathfrak p}\otimes_{A_{\mathfrak p}}H(\mathfrak p),\qquad
	h\otimes\overline a^{\op}\longmapsto\overline a\otimes h,
	\]
	we obtain a flat cover
	\begin{equation}\label{eq:Hpcover}
		H(\mathfrak p)\longrightarrow\overline A_{\mathfrak p}\otimes_{A_{\mathfrak p}}H(\mathfrak p).
	\end{equation}
	The target of \eqref{eq:Hpcover} cannot be zero. If it were zero, then the zero map $H(\mathfrak p)\to0$ would be a flat cover. But the zero endomorphism of the nonzero module $H(\mathfrak p)$ satisfies $0\circ0=0$ and is not an automorphism, contradicting right minimality. Therefore
	\[
	\overline A_{\mathfrak p}\otimes_{A_{\mathfrak p}}H(\mathfrak p)\ne0.
	\]
	By \eqref{eq:Tidentify}, $T_{\mathfrak p}(H)\ne0$.
 \end{proof}

We shall also use a standard splitting argument.

\begin{lemma}\label{lem:splitting}
	Let $f:F\to N$ be a flat precover and let $g:G\to N$ be a flat cover. Then there is a decomposition
	\[
	F\cong G\oplus H
	\]
	under which
	\[
	f=(g,0):G\oplus H\longrightarrow N.
	\]
\end{lemma}
\begin{proof} It follows by \cite[Lemma 5.8]{GT12}.
\end{proof}

\begin{proposition}[Local minimality criterion]\label{prop:local-min}
	Let $f:F\to N$ be a flat precover of left $A$-modules. Assume that $F$ is flat cotorsion. If $T_{\mathfrak p}(f)$ is a monomorphism for every $\mathfrak p\in\Spec R$, then $f$ is a flat cover.
\end{proposition}
\begin{proof}
	Let $g:G\to N$ be a flat cover; such a cover exists by \cite{BicanElBashirEnochs2001}. By Lemma~\ref{lem:splitting}, there is a decomposition
	\[
	F\cong G\oplus H
	\]
	under which $f=(g,0)$. Since $H$ is a direct summand of the flat cotorsion module $F$, the module $H$ is again flat and cotorsion.
	
	Assume $H\ne0$. By Lemma~\ref{lem:nonvanishing-top}, there exists $\mathfrak p\in\Spec R$ for which $T_{\mathfrak p}(H)\ne0$. Both functors $\Hom_R(R_{\mathfrak p},-)$ and $\overline A_{\mathfrak p}\otimes_{A_{\mathfrak p}}-$ preserve finite direct sums. Consequently
	\[
	T_{\mathfrak p}(F)\cong T_{\mathfrak p}(G)\oplus T_{\mathfrak p}(H).
	\]
	Since $f$ vanishes on $H$, the nonzero direct summand $T_{\mathfrak p}(H)$ is contained in $\Ker T_{\mathfrak p}(f)$, contradicting the assumed injectivity of $T_{\mathfrak p}(f)$. Hence $H=0$. Thus $f$ is isomorphic, as a morphism to $N$, to the flat cover $g$, and therefore $f$ is a flat cover.
\end{proof}

\begin{remark}
	The assertion of Proposition~\ref{prop:local-min} is deliberately one-sided. For the proof of the main theorem we only need that injectivity of all local top maps removes every redundant flat-cotorsion summand. No converse local criterion for arbitrary noncommutative Noether algebras is asserted here.
\end{remark}
The following example shows that the flat-cotorsion assumption in Proposition~\ref{prop:local-min} cannot be dropped.
\begin{example}
	Let $R=A=\mathbb Z$ and consider the zero morphism
	\[
	f:\mathbb Z\longrightarrow0.
	\]
	It is a flat precover, since $\mathbb Z$ is flat and every map to the zero module factors through $f$. For every $\mathfrak p\in\Spec\mathbb Z$ one has
	\[
	\Hom_{\mathbb Z}(\mathbb Z_{\mathfrak p},\mathbb Z)=0.
	\]
	Indeed, if $\mathfrak p=(0)$, the value of a homomorphism on $1$ must be divisible by every positive integer. If $\mathfrak p=(\ell)$, choose an integer $s>1$ not divisible by $\ell$; the value on $1$ must then be divisible by $s^n$ for every $n\ge1$. In either case the value is zero, and a homomorphism $\mathbb Z_{\mathfrak p}\to\mathbb Z$ is determined by that value. Hence
	\[
	T_{\mathfrak p}(\mathbb Z)=0
	\]
	for every $\mathfrak p$, and every $T_{\mathfrak p}(f):0\to0$ is monic. But $f$ is not a flat cover, because the zero endomorphism of the nonzero source $\mathbb Z$ satisfies $f0=f$ and is not invertible. Thus the flat-cotorsion assumption in Proposition~\ref{prop:local-min} cannot be dropped. Consistently, $\mathbb Z$ is not cotorsion; classically $\Ext^1_{\mathbb Z}(\mathbb Q,\mathbb Z)\ne0$; see \cite[Section~54]{Fuchs1970}.
\end{example}

\section{Duality between injective envelopes and flat covers}
We can now prove the main result.

\begin{theorem}[Noether-algebra extension of Puuska's theorem]\label{thm:main}
	Let $R$ be a commutative noetherian ring, let $A$ be a Noether $R$-algebra, and let $i:M\to I$ be a morphism of right $A$-modules. The following conditions are equivalent:
	\begin{enumerate}[label=(\arabic*)]
		\item $i$ is an injective envelope in $\Mod A$;
		\item for some injective cogenerator $E$ of $\Mod R$, the morphism
		\[
		D_E(i)=\Hom_R(i,E):\Hom_R(I,E)\longrightarrow\Hom_R(M,E)
		\]
		is a flat cover in $\Mod A^{\op}$;
		\item for every injective $R$-module $E$, the morphism $D_E(i)$ is a flat cover in $\Mod A^{\op}$.
	\end{enumerate}
\end{theorem}
\begin{proof}
	$(1)\Rightarrow(3)$. Assume that $i$ is an injective envelope and let $E$ be an arbitrary injective $R$-module. Put
	\[
	F=D_E(I),\qquad N=D_E(M),\qquad f=D_E(i).
	\]
	Let $C=\Coker i$. Since $i$ is monic, we have a short exact sequence
	\[
	0\longrightarrow M\xrightarrow{i}I\longrightarrow C\longrightarrow0.
	\]
	Exactness of $D_E$ gives a short exact sequence of left $A$-modules
	\begin{equation}\label{eq:dualshort}
		0\longrightarrow D_E(C)\longrightarrow F\xrightarrow{f}N\longrightarrow0.
	\end{equation}
	By (D2), $F=D_E(I)$ is flat cotorsion. By (D1), $D_E(C)$ is pure-injective and hence cotorsion.
	
	We claim that $f$ is a flat precover. Let $L$ be a flat left $A$-module. Applying $\Hom_{A^{\op}}(L,-)$ to \eqref{eq:dualshort} gives the exact segment
	\[
	\Hom_{A^{\op}}(L,F)\longrightarrow\Hom_{A^{\op}}(L,N)
	\longrightarrow\Ext^1_{A^{\op}}(L,D_E(C)).
	\]
	The last term vanishes because $D_E(C)$ is cotorsion. Hence every map $L\to N$ factors through $f$, proving the claim.
	
	By Proposition~\ref{prop:local-inj}, the morphism
	\[
	S_{\mathfrak p}(i):S_{\mathfrak p}(M)\longrightarrow S_{\mathfrak p}(I)
	\]
	is an isomorphism for every $\mathfrak p\in\Spec R$. The naturality of Lemma~\ref{lem:socle-top} gives a commutative square
	\[
	\begin{CD}
		T_{\mathfrak p}(D_EI) @>{\cong}>> D_E(S_{\mathfrak p}(I))\\
		@V{T_{\mathfrak p}(f)}VV @VV{D_E(S_{\mathfrak p}(i))}V\\
		T_{\mathfrak p}(D_EM) @>{\cong}>> D_E(S_{\mathfrak p}(M)).
	\end{CD}
	\]
	The right vertical map is an isomorphism because $S_{\mathfrak p}(i)$ is an isomorphism. Hence $T_{\mathfrak p}(f)$ is an isomorphism, in particular a monomorphism, for every $\mathfrak p$. Since $F$ is flat cotorsion and $f$ is a flat precover, Proposition~\ref{prop:local-min} shows that $f$ is a flat cover. Thus (3) holds.
	
	$(3)\Rightarrow(2)$. The category $\Mod R$ has an injective cogenerator. Applying (3) to any such injective cogenerator gives (2).
	
	$(2)\Rightarrow(1)$. Let $E$ be an injective cogenerator such that
	\[
	f=D_E(i):D_E(I)\longrightarrow D_E(M)
	\]
	is a flat cover. A flat cover is an epimorphism and its source is flat. Thus $D_E(I)$ is flat. By (D3), $I$ is injective.
	
	We next prove that $i$ is a monomorphism. Let $K=\Ker i$. Since $D_E$ is exact, the exact sequence associated to $i$ gives
	\[
	\Coker f\cong D_E(K).
	\]
	Because $f$ is an epimorphism, $D_E(K)=0$. As $E$ is a cogenerator, $\Hom_R(K,E)=0$ implies $K=0$. Hence $i$ is monic.
	
	It remains to prove that $i$ is left minimal. Let $a\in\End_A(I)$ satisfy $ai=i$. Contravariance gives
	\[
	fD_E(a)=D_E(i)D_E(a)=D_E(ai)=D_E(i)=f.
	\]
	Since $f$ is a flat cover, right minimality implies that $D_E(a)$ is an automorphism. We now use the cogenerator hypothesis to reflect this isomorphism. Because $D_E$ is exact, if $D_E(a)$ is an isomorphism then
	\[
	D_E(\Ker a)=0\qquad\text{and}\qquad D_E(\Coker a)=0.
	\]
	Since $E$ is a cogenerator, this forces $\Ker a=0=\Coker a$. Thus $a$ is an automorphism. Hence $i$ is a left-minimal monomorphism into an injective module.
	
	For completeness, we verify that such a morphism is an injective envelope. Let $e:M\to E_A(M)$ be the injective envelope. Since $I$ is injective, there is $u:E_A(M)\to I$ with $ue=i$. As in the proof of Proposition~\ref{prop:local-inj}, essentiality of $e(M)$ and monicity of $i$ imply that $u$ is monic. Since $E_A(M)$ is injective, $u$ splits; choose $r:I\to E_A(M)$ with $ru=1_{E_A(M)}$. Then $ur\in\End_A(I)$ is idempotent and
	\[
	(ur)i=urue=ue=i.
	\]
	Left minimality of $i$ implies that $ur$ is an automorphism. An idempotent automorphism is the identity, so $ur=1_I$. Thus $u$ is also surjective and hence is an isomorphism. Therefore $i$ is isomorphic to the injective envelope $e$, and (1) follows.
\end{proof}

\begin{remark}
	The full local-complete structure theory is needed only in the implication $(1)\Rightarrow(3)$, more precisely in Lemma~\ref{lem:nonvanishing-top}, where a nonzero flat cotorsion direct summand is detected by one of the local top functors $T_{\mathfrak p}$. The reverse implication for an injective cogenerator is formal once one knows the flatness--injectivity duality in \cite[Proposition~2.2]{KandaNakamura2022}.
	
	If $A=R$ is commutative, then
	\[
	\overline A_{\mathfrak p}=R_{\mathfrak p}/\mathfrak pR_{\mathfrak p}=\kappa(\mathfrak p),
	\]
	and the functors $S_{\mathfrak p}$ and $T_{\mathfrak p}$ are precisely the local functors used in Puuska's proof. Thus Theorem~\ref{thm:main} recovers \cite[Theorem~1.1]{Puuska2024}.
\end{remark}

\section{Minimal resolutions and special cases}
We first state the resolution-theoretic consequence with the successive cosyzygies made explicit. Here a flat resolution
\[
\cdots\longrightarrow F_2\longrightarrow F_1\longrightarrow F_0\longrightarrow N\longrightarrow0
\]
is called minimal if, at every degree, the induced epimorphism from $F_n$ onto the corresponding syzygy (with $N$ as the zeroth syzygy) is a flat cover. This is the standard cover-theoretic convention of \cite[Chapter~1]{Xu1996}.

\begin{corollary}[Minimal resolutions]\label{cor:minres}
	Let
	\[
	0\longrightarrow M\longrightarrow I^0\longrightarrow I^1\longrightarrow I^2\longrightarrow\cdots
	\]
	be an injective resolution of a right $A$-module $M$. The following are equivalent:
	\begin{enumerate}[label=(\arabic*)]
		\item the injective resolution is minimal;
		\item for some injective cogenerator $E$ of $\Mod R$, the exact complex
		\[
		\cdots\longrightarrow D_E(I^2)\longrightarrow D_E(I^1)\longrightarrow D_E(I^0)
		\longrightarrow D_E(M)\longrightarrow0
		\]
		is a minimal flat resolution in $\Mod A^{\op}$;
		\item  the exact complex
		\[
		\cdots\longrightarrow D_E(I^2)\longrightarrow D_E(I^1)\longrightarrow D_E(I^0)
		\longrightarrow D_E(M)\longrightarrow0
		\]
		is a minimal flat resolution in $\Mod A^{\op}$ for every injective $R$-module $E$.
	\end{enumerate}
\end{corollary}
\begin{proof}
	Define the cosyzygies by $C^0=M$ and, for $n\ge0$, let $C^{n+1}$ be the cokernel of the canonical inclusion $C^n\hookrightarrow I^n$. Thus for every $n\ge0$ there is a short exact sequence
	\begin{equation}\label{eq:cosyz}
		0\longrightarrow C^n\xrightarrow{j_n}I^n\longrightarrow C^{n+1}\longrightarrow0.
	\end{equation}
	The injective resolution is minimal if and only if every $j_n$ is an injective envelope.
	
	Let $E$ be injective over $R$. Exactness of $D_E$ transforms \eqref{eq:cosyz} into
	\begin{equation}\label{eq:dual-cosyz}
		0\longrightarrow D_E(C^{n+1})\longrightarrow D_E(I^n)\xrightarrow{D_E(j_n)}D_E(C^n)\longrightarrow0.
	\end{equation}
	In the dual complex, the image of the differential $D_E(I^n)\to D_E(I^{n-1})$ is exactly the embedded copy of $D_E(C^n)$ appearing in the dual of the preceding cosyzygy sequence. Indeed, the original differential $I^{n-1}\to I^n$ factors as
	\[
	I^{n-1}\twoheadrightarrow C^n\xrightarrow{j_n}I^n,
	\]
	so its dual factors as
	\[
	D_E(I^n)\xrightarrow{D_E(j_n)}D_E(C^n)\hookrightarrow D_E(I^{n-1}),
	\]
	and $D_E(j_n)$ is epic by \eqref{eq:dual-cosyz}. At degree zero the corresponding target is $D_E(C^0)=D_E(M)$.
	
	Consequently, the map from the $n$-th flat term onto the $n$-th image of the dual resolution is exactly $D_E(j_n)$. By Theorem~\ref{thm:main}, $j_n$ is an injective envelope if and only if $D_E(j_n)$ is a flat cover for some injective cogenerator, equivalently for every injective $R$-module. Applying this degree by degree proves the equivalence of (1), (2), and (3).
\end{proof}

\appendix
\setcounter{equation}{0}
\renewcommand{\theequation}{A.\arabic{equation}}
\section{ The classical proof of Artin algebra case}\label{app:artin}
Throughout this appendix, $R$ is a commutative Artinian ring and $A$ is an \emph{Artin $R$-algebra}: the image of $R$ lies in $Z(A)$ and $A$ is finitely generated as an $R$-module. In this appendix, we will gvie a self-contained proof of Theorem \ref{thm:main} for the  Artin-algebra case case by the classical methods. Put
\[
\mathfrak r=J(R),\qquad J=\rad A,\qquad \overline A=A/J.
\]
Then $A$ is left and right Artinian (hence left and right Noetherian), and $J$ is nilpotent. Let
\[
E_R=E_R(R/\mathfrak r)
\]
be the injective envelope of the semisimple $R$-module $R/\mathfrak r$. It is an injective cogenerator of $\Mod R$. The standard Artin--Matlis dual is
\begin{equation}\label{eq:artinD}
	D(-)=\Hom_R(-,E_R).
\end{equation}
More generally, for an injective $R$-module $Q$ we write
\[
D_Q(-)=\Hom_R(-,Q).
\]
Unless finite generation is explicitly imposed, modules in this appendix are allowed to be arbitrary. For a right $A$-module $X$ and a left $A$-module $Y$, set
\begin{equation}\label{eq:artin-soc-top}
	\Soc_A X=\Hom_A(\overline A,X)\cong\{x\in X\mid xJ=0\},
	\qquad
	\Top_A Y=\overline A\otimes_A Y\cong Y/JY.
\end{equation}
We write $\mathrm{mod}\,A$ for finitely generated right $A$-modules and $A\text{-}\mathrm{mod}$ for finitely generated left $A$-modules. Since $R$ is Artinian and $A$ is module-finite over $R$, these are precisely the finite-length $A$-modules on the respective sides. The standard Artin-algebra facts about injective envelopes, projective covers, and the duality \eqref{eq:artinD} may be found in \cite[Chapters~10 and 27]{AndersonFuller1992}.

\subsection{The base-ring duality and the socle--top formula}
We begin by recording the exactness and faithfulness that replace the elementary vector-space facts used over a field.

\begin{lemma}[Duality over Artinian rings]\label{lem:artin-base-duality}
	Let $Q$ be an injective $R$-module.
	\begin{enumerate}[label=(\roman*)]
		\item $D_Q=\Hom_R(-,Q)$ is exact and contravariant.
		\item If $Q$ is an injective cogenerator, then $D_Q$ is faithful and reflects zero objects and isomorphisms.
		\item For $Q=E_R$, the evaluation morphism $\delta_L:L\to D^2L$ is an isomorphism for every finite-length $R$-module $L$. Consequently
		\[
		D:\mathrm{mod}\,A\xrightarrow{\ \sim\ }(A\text{-}\mathrm{mod})^{\op}
		\]
		is an exact contravariant duality.
	\end{enumerate}
\end{lemma}
\begin{proof}
	Part (i) is the defining exactness property of an injective module.
	
	For (ii), suppose $0\ne L$ is an $R$-module. Since $R$ is Artinian, its Jacobson radical $\mathfrak r$ is nilpotent. Choose $0\ne x\in L$ and choose $m$ maximal with $x\mathfrak r^m\ne0$. Then $x\mathfrak r^m$ is annihilated by $\mathfrak r$, so it is a nonzero module over the semisimple ring $R/\mathfrak r$ and therefore contains a simple submodule $S$. Because $Q$ is a cogenerator, there is a nonzero map $S\to Q$; injectivity of $Q$ extends it along $S\hookrightarrow L$ to a nonzero map $L\to Q$. Thus $D_Q(L)\ne0$, proving faithfulness and reflection of zero objects. Since $D_Q$ is exact, if $D_Q(f)$ is an isomorphism then $D_Q(\Ker f)=D_Q(\Coker f)=0$, whence $\Ker f=\Coker f=0$; thus $f$ is an isomorphism.
	
	For (iii), note first that $R/\mathfrak r$ is a finite direct sum of representatives of the simple $R$-modules, and $E_R=E_R(R/\mathfrak r)$ is the corresponding finite direct sum of their injective envelopes. If $S$ is simple, then $D(S)=\Hom_R(S,E_R)$ is simple and the evaluation map $S\to D^2S$ is nonzero; hence it is an isomorphism. Now let $L$ have finite length. An induction on the length, using a short exact sequence $0\to L'\to L\to S\to0$ with $S$ simple, exactness of $D$, and the five lemma applied to the evaluation diagram, shows that $\delta_L:L\to D^2L$ is an isomorphism. Every finitely generated $A$-module has finite $R$-length, because $A$ is finite over the Artinian ring $R$. Therefore $D$ carries finitely generated modules to finitely generated modules on the opposite side and $D^2\cong1$ there, giving the asserted duality.
\end{proof}

\begin{lemma}[Classical socle--top duality]\label{lem:artin-soc-top}
	Let $Q$ be an injective $R$-module and let $X$ be a right $A$-module. Evaluation induces a natural isomorphism
	\begin{equation}\label{eq:artin-soc-top-dual}
		\Top_A(D_QX)\xrightarrow{\ \cong\ }D_Q(\Soc_A X).
	\end{equation}
	For $Q=E_R$, this is the standard Artin-algebra socle--top duality.
\end{lemma}
\begin{proof}
	Define
	\[
	\Theta_X:\overline A\otimes_A D_Q(X)\longrightarrow D_Q(\Hom_A(\overline A,X))
	\]
	by
	\begin{equation}\label{eq:artin-theta}
		\Theta_X(\overline a\otimes\lambda)(g)=\lambda(g(\overline a)).
	\end{equation}
	If $b\in A$, then
	\[
	\Theta_X(\overline a b\otimes\lambda)(g)
	=\lambda(g(\overline a)b)
	=(b\lambda)(g(\overline a))
	=\Theta_X(\overline a\otimes b\lambda)(g),
	\]
	where $(b\lambda)(x)=\lambda(xb)$. Thus \eqref{eq:artin-theta} is balanced.
	
	Since $A$ is right Noetherian, the cyclic right $A$-module $\overline A=A/J$ is finitely presented. Choose a finite presentation
	\[
	A^r\longrightarrow A^s\longrightarrow\overline A\longrightarrow0.
	\]
	Applying $\Hom_A(-,X)$ and then the exact functor $D_Q$ gives a right-exact sequence whose first two terms are $D_Q(X)^r$ and $D_Q(X)^s$. Tensoring the presentation with $D_Q(X)$ gives a right-exact sequence with the same first two terms and the same matrix map, because $(b\lambda)(x)=\lambda(xb)$. The induced map on cokernels is precisely $\Theta_X$, so it is an isomorphism. Naturality is immediate from the evaluation formula.
\end{proof}

\subsection{Injectives and projectives under Artin duality}
The following form is valid for arbitrary modules and avoids any unjustified use of reflexivity outside the finite-length subcategories.

\begin{lemma}[Injective--projective correspondence]\label{lem:artin-inj-proj}
	Let $Q$ be an injective $R$-module and let $X$ be a right $A$-module.
	\begin{enumerate}[label=(\roman*)]
		\item If $X$ is injective, then $D_Q(X)$ is a projective left $A$-module.
		\item If $Q$ is an injective cogenerator and $D_Q(X)$ is flat (equivalently, projective), then $X$ is injective.
		\item In particular,
		\[
		X\text{ is injective}\quad\Longleftrightarrow\quad D(X)\text{ is projective}.
		\]
	\end{enumerate}
\end{lemma}
\begin{proof}
	Let $L$ be a right ideal of $A$. Since $A$ is right Noetherian, $L$ is finitely presented. The same finite-presentation argument as in Lemma~\ref{lem:artin-soc-top} gives a natural isomorphism
	\begin{equation}\label{eq:ideal-dual}
		L\otimes_A D_Q(X)\xrightarrow{\ \cong\ }D_Q(\Hom_A(L,X)).
	\end{equation}
	Under \eqref{eq:ideal-dual}, the morphism induced by $L\hookrightarrow A$ corresponds to the dual of the restriction map
	\begin{equation}\label{eq:restriction}
		\rho_L:\Hom_A(A,X)\longrightarrow\Hom_A(L,X).
	\end{equation}
	If $X$ is injective, \eqref{eq:restriction} is surjective. Exactness of $D_Q$ makes its dual injective. Hence
	\[
	L\otimes_A D_Q(X)\longrightarrow A\otimes_A D_Q(X)
	\]
	is monic for every right ideal $L$. The ideal criterion for flatness \cite[Chapter~3]{EnochsJenda2000} shows that $D_Q(X)$ is flat. Since $A$ is left Artinian, it is left perfect by Bass's characterization \cite{Bass1960}; hence every flat left $A$-module is projective \cite[Theorem~28.4 and Corollary~28.8]{AndersonFuller1992}. This proves (i).
	
	Conversely, assume that $Q$ is an injective cogenerator and $D_Q(X)$ is flat. Then the tensor map associated to every right ideal $L$ is monic. Via \eqref{eq:ideal-dual}, this says that $D_Q(\Coker\rho_L)=0$. Faithfulness of $D_Q$ gives $\Coker\rho_L=0$. Thus every map $L\to X$ extends to $A$, and Baer's criterion makes $X$ injective \cite[Chapter~3]{EnochsJenda2000}. Part (iii) follows by taking $Q=E_R$ and using that $E_R$ is an injective cogenerator.
\end{proof}

\begin{remark}
	For a finitely generated right $A$-module $X$, Lemma~\ref{lem:artin-base-duality}(iii) gives $X\cong D^2X$. For an arbitrary module this reflexivity need not hold; Lemma~\ref{lem:artin-inj-proj}, rather than bidual reflexivity, is the correct tool for the arbitrary-module form of the theorem below.
\end{remark}

\subsection{Classical envelope and cover criteria}
\begin{lemma}\label{lem:artin-socle-criterion}
	Let $i:M\to I$ be a morphism of right $A$-modules. Then $i$ is an injective envelope if and only if $I$ is injective, $i$ is monic, and
	\begin{equation}\label{eq:soclecriterion}
		\Soc_A(i):\Soc_A M\longrightarrow\Soc_A I
	\end{equation}
	is an isomorphism.
\end{lemma}
\begin{proof}
	Identify $M$ with its image when $i$ is monic. If $M$ is essential in $I$, every simple submodule $S\subseteq I$ satisfies $S\cap M\ne0$, hence $S\subseteq M$. Therefore $\Soc_A I\subseteq M$ and \eqref{eq:soclecriterion} is an isomorphism.
	
	Conversely, suppose $I$ is injective, $i$ is monic, and $\Soc_A M=\Soc_A I$. Every nonzero right $A$-module $N$ has nonzero socle. Indeed, choose $0\ne x\in N$ and choose $r$ maximal with $xJ^r\ne0$; then $xJ^r$ is a nonzero module annihilated by $J$, hence a nonzero module over the semisimple ring $A/J$, and it contains a simple submodule. Thus every nonzero submodule $N\subseteq I$ contains a simple submodule of $I$. That simple module lies in $\Soc_A I=\Soc_A M\subseteq M$, so $N\cap M\ne0$. Hence $M$ is essential in $I$, and $i$ is an injective envelope.
\end{proof}

\begin{lemma}\label{lem:artin-top-criterion}
	Let $f:P\to N$ be a morphism of left $A$-modules. Then $f$ is a projective cover if and only if $P$ is projective, $f$ is epic, and
	\begin{equation}\label{eq:topcriterion}
		\Top_A(f):P/JP\longrightarrow N/JN
	\end{equation}
	is an isomorphism.
\end{lemma}
\begin{proof}
	Let $K=\Ker f$ and assume first that $P$ is projective and $f$ is epic. We recall two elementary facts.
	
	First, $f$ is a projective cover if and only if $K\smallin P$. If $K\smallin P$ and $u\in\End_A(P)$ satisfies $fu=f$, then $P=\im u+K$, so $u$ is epic. Projectivity of $P$ splits $u$, giving $P=\Ker u\oplus P'$ for a submodule $P'$. Moreover, $\Ker u\subseteq K$; a superfluous direct summand must be zero. Thus $u$ is an automorphism. Conversely, if $f$ is right minimal and $K+L=P$, then $f|_L:L\to N$ is epic. Projectivity gives $v:P\to L$ with $(f|_L)v=f$. If $j:L\hookrightarrow P$ is inclusion, then $f(jv)=f$; right minimality makes $jv$ invertible, and therefore $L=P$. Hence $K\smallin P$.
	
	Second, $JP\smallin P$. If $L+JP=P$, then $P/L=J(P/L)$. Iterating and using $J^n=0$ for some $n$ gives $P/L=J^n(P/L)=0$, so $L=P$.
	
	The map in \eqref{eq:topcriterion} is always epic, and its kernel is
	\begin{equation}\label{eq:topkernel}
		(K+JP)/JP,
	\end{equation}
	because $f(JP)=Jf(P)=JN$. Consequently \eqref{eq:topcriterion} is an isomorphism exactly when $K\subseteq JP$. If this inclusion holds, then $K\smallin P$ because $JP\smallin P$, and $f$ is a projective cover.
	
	For the converse, suppose $f$ is a projective cover. Then $K\smallin P$. If the image $(K+JP)/JP$ in the semisimple module $P/JP$ were nonzero, choose a complementary submodule $\overline L$ with
	\[
	P/JP=((K+JP)/JP)\oplus\overline L.
	\]
	Let $L$ be the inverse image of $\overline L$ in $P$. Then $K+L=P$, while $L\ne P$, contradicting $K\smallin P$. Hence $K\subseteq JP$, so \eqref{eq:topcriterion} is an isomorphism.
\end{proof}

\subsection{The Artin-algebra form of the main theorem for arbitrary modules}
\begin{theorem}[Artin-algebra form of Theorem~\ref{thm:main}]\label{thm:artin-main}
	Let $R$ be a commutative Artinian ring, let $A$ be an Artin $R$-algebra, and let $i:M\to I$ be a morphism of arbitrary right $A$-modules. Let $D=\Hom_R(-,E_R)$ with $E_R=E_R(R/J(R))$. The following are equivalent:
	\begin{enumerate}[label=(\arabic*)]
		\item $i$ is an injective envelope;
		\item $D(i):D(I)\to D(M)$ is a projective cover;
		\item for some injective cogenerator $Q$ of $\Mod R$, the morphism $D_Q(i):D_Q(I)\to D_Q(M)$ is a projective cover;
		\item for every injective $R$-module $Q$, the morphism $D_Q(i)$ is a projective cover.
	\end{enumerate}
\end{theorem}
\begin{proof}
	Assume (1), and let $Q$ be any injective $R$-module. Exactness of $D_Q$ makes $D_Q(i)$ epic. By Lemma~\ref{lem:artin-inj-proj}(i), $D_Q(I)$ is projective. Lemma~\ref{lem:artin-socle-criterion} makes $\Soc_A(i)$ an isomorphism. Naturality of Lemma~\ref{lem:artin-soc-top} gives a commutative square
	\[
	\begin{CD}
		\Top_A(D_QI) @>{\cong}>> D_Q(\Soc_A I)\\
		@V{\Top_A(D_Qi)}VV @VV{D_Q(\Soc_A i)}V\\
		\Top_A(D_QM) @>{\cong}>> D_Q(\Soc_A M).
	\end{CD}
	\]
	The right vertical map is an isomorphism, so the left vertical map is an isomorphism. Lemma~\ref{lem:artin-top-criterion} shows that $D_Q(i)$ is a projective cover. This proves $(1)\Rightarrow(4)$.
	
	The implication $(4)\Rightarrow(2)$ follows by taking $Q=E_R$, and (2) implies (3) by the same choice because $E_R$ is an injective cogenerator. It remains to prove $(3)\Rightarrow(1)$. Let $Q$ be an injective cogenerator and assume $D_Q(i)$ is a projective cover. Its source is projective, hence flat, so Lemma~\ref{lem:artin-inj-proj}(ii) implies that $I$ is injective. Since $D_Q(i)$ is epic and $D_Q$ is exact,
	\[
	D_Q(\Ker i)=0.
	\]
	Faithfulness of $D_Q$ gives $\Ker i=0$, so $i$ is monic.
	
	By Lemma~\ref{lem:artin-top-criterion}, $\Top_A(D_Qi)$ is an isomorphism. The naturality square above and Lemma~\ref{lem:artin-soc-top} imply that $D_Q(\Soc_A i)$ is an isomorphism. Exactness and faithfulness of $D_Q$ reflect isomorphisms by Lemma~\ref{lem:artin-base-duality}(ii), so $\Soc_A(i)$ is an isomorphism. The socle criterion, Lemma~\ref{lem:artin-socle-criterion}, now proves that $i$ is an injective envelope.
\end{proof}

For finitely generated modules the preceding theorem admits the customary purely categorical Artin-algebra proof.

\begin{corollary}[Classical finite-length proof]\label{cor:finite-artin}
	Let $M$ and $I$ be finitely generated right $A$-modules. Then
	\begin{equation}\label{eq:finite-artin-cover}
		i:M\longrightarrow I\text{ is an injective envelope}
		\quad\Longleftrightarrow\quad
		D(i):D(I)\longrightarrow D(M)\text{ is a projective cover}.
	\end{equation}
\end{corollary}
\begin{proof}
	By Lemma~\ref{lem:artin-base-duality}(iii),
	\[
	D:\mathrm{mod}\,A\xrightarrow{\ \sim\ }(A\text{-}\mathrm{mod})^{\op}
	\]
	is exact, full, faithful, and satisfies $D^2\cong1$. It sends injective objects to projective objects and monomorphisms to epimorphisms.
	
	Suppose $i$ is an injective envelope and $u\in\End_{A^{\op}}(D(I))$ satisfies $D(i)u=D(i)$. Fullness gives a unique $a\in\End_A(I)$ with $u=D(a)$. Faithfulness transforms the displayed identity into $ai=i$. Left minimality of $i$ makes $a$ invertible, hence $u=D(a)$ is invertible. Therefore $D(i)$ is a projective cover.
	
	Conversely, if $D(i)$ is a projective cover, exactness and $D^2\cong1$ show that $i$ is a monomorphism into an injective module. If $a\in\End_A(I)$ satisfies $ai=i$, then $D(i)D(a)=D(i)$; right minimality makes $D(a)$ invertible. Faithfulness and $D^2\cong1$ make $a$ invertible. Thus $i$ is left minimal, and a left-minimal monomorphism into an injective module is an injective envelope.
\end{proof}

\end{document}